\documentclass[11pt]{amsart}
\usepackage{amsmath,amsthm,amsfonts,amssymb,times,latexsym,mathrsfs,mathabx,bbm}

\usepackage{hyperref,url}

\usepackage{color}
\usepackage[usenames,dvipsnames,svgnames,table]{xcolor}

\numberwithin{equation}{section}  

\DeclareMathAlphabet{\curly}{U}{rsfs}{m}{n}  

\theoremstyle{remark}

\theoremstyle{plain}
\newtheorem{prop}{Proposition}
\newtheorem{lem}{Lemma}[section]
\newtheorem{thm}{Theorem}
\newtheorem{cor}{Corollary}

\numberwithin{equation}{section}

\newcommand{\ZZ}{{\mathbb Z}}

\makeatletter
\renewcommand{\pmod}[1]{\allowbreak\mkern7mu({\operator@font mod}\,\,#1)}
\makeatother

\newcommand{\be}{\begin{equation}}
\newcommand{\ee}{\end{equation}}

\newcommand{\eps}{\ensuremath{\varepsilon}}

\renewcommand{\le}{\leqslant}

\renewcommand{\ge}{\geqslant}
\renewcommand{\geq}{\geqslant}

\renewcommand{\(}{\left(}
\renewcommand{\)}{\right)}

\newcommand{\pfrac}[2]{\left(\frac{#1}{#2}\right)}  

\begin{document}

\title{Explicit RIP matrices: an update}
\author{Kevin Ford}
\address{Department of Mathematics, 1409 West Green Street, University
of Illinois at Urbana-Champaign, Urbana, IL 61801, USA}
\email{ford@math.uiuc.edu}
\author{Denka Kutzarova}
\address{Department of Mathematics, 1409 West Green Street, University
of Illinois at Urbana-Champaign, Urbana, IL 61801, USA}
\email{denka@math.uiuc.edu}
\author{George Shakan}
\address{Department of Mathematics, University of Oxford, Radcliffe Observatory, Andrew Wiles Building, Woodstock Rd, Oxford OX2 6GG, UK}
\email{george.shakan@gmail.com}

\begin{abstract}
Leveraging recent advances in additive combinatorics, we exhibit explicit
matrices satisfying the Restricted Isometry Property with better parameters.
Namely, for $\eps=3.26\cdot 10^{-7}$, large $k$ and
$k^{2-\varepsilon} \le N\le k^{2+\varepsilon}$, we construct  $n \times N$ RIP matrices
of order $k$ with $k = \Omega( n^{1/2+\eps/4} )$.
\end{abstract}


\date{\today}

\maketitle


%
\section{Introduction}
%

Suppose $1\le k \le  n\le N$ and $0<\delta<1$. A `signal' ${\mathbf x} =
(x_j)_{j=1}^N$ is said to be $k$-sparse if ${\mathbf x}$
has at most $k$ nonzero coordinates. An $n \times N$   matrix
$\Phi$ is said to satisfy the
 Restricted Isometry Property (RIP) of order $k$ with constant
 $\delta$ if for all $k$-sparse vectors ${\mathbf x}$ we have
\begin{equation}\label{eq: RIP}
(1 - \delta)\|{\mathbf x}\|_2^2 \le \|\Phi {\mathbf x}\|_2^2 \le (1 + \delta)
\|{\mathbf x}\|_2^2.
\end{equation}
While most authors work with real signals and matrices, in this paper
we work with complex matrices for convenience.  Given a complex
matrix $\Phi$ satisfying \eqref{eq: RIP},  the $2n \times 2N$ real
matrix $\Phi'$, formed by replacing each
element $a+ib$ of $\Phi$ by the $2 \times 2$ matrix
$(\begin{smallmatrix} a&b\\ -b&a \end{smallmatrix})$, also
satisfies \eqref{eq: RIP} with the same parameters $k,\delta$.

We know from Cand\`es, Romberg and Tao that
matrices satisfying RIP have application to sparse signal recovery
(see \cite{Ca, CRT,CT}). 
Given $n,N,\delta$, we wish to find $n \times N$ RIP matrices of
order $k$ with constant $\delta$, and with $k$ as large as
possible.  If the entries of $\Phi$ are independent Bernoulli
random variables with values $\pm 1/\sqrt{n}$, then with high
probability, $\Phi$ will have the required properties
for $k$ of order close to $\delta n$;
in different language, this was first proved by
Kashin \cite{Ka2}.

It is an open problem to find good \emph{explicit}
constructions of RIP matrices; see Tao's Weblog \cite{Tao} for a
discussion of the problem.  
All existant explicit constructions of RIP matrices are based on
number theory.  
Prior to the work of Bourgain, Dilworth, Ford, Konyagin and Kutzarova
\cite{BDFKK},
there were many constructions, e.g. Kashin \cite{Ka}, DeVore
\cite{D} and Nelson and Temlyakov \cite{NT},
 producing matrices with
$\delta$ small and order
\begin{equation}\label{kD}
k \approx \delta\frac{\sqrt{n}\log n}{\log N}.
\end{equation}

The $\sqrt{n}$ barrier was broken by the aforementioned authors in
\cite{BDFKK}: 

\bigskip

\textit{Theorem A. \cite{BDFKK}}.
There are effective constants $\varepsilon>0$, $\varepsilon'>0$ and explicit numbers
$k_0,c>0$ such that for any positive integers $k\geq k_0$ and
$k^{2-\varepsilon} \le N\le  k^{2+\varepsilon}$, there is an explicit $n \times N$ RIP matrix of order
$k$ with $k\ge c n^{1/2+\eps/4}$ and constant $\delta=k^{-\varepsilon'}$.

\bigskip

As reported in \cite{STOC}, the construction in \cite{BDFKK} produces
a value $\eps \approx 2 \cdot 10^{-22}$.  An improved construction was
presented in \cite{STOC}, giving Theorem A with $\eps=3.6\cdot 10^{-15}$.
The values of $\eps$ depend on two constants in additive combinatorics,
which have since been improved.  
Incorporating these improvements into the argument in \cite{STOC}, we will deduce the following.

\begin{thm}\label{thm1}
Let $\eps=3.26\cdot 10^{-7}$.  There is $\varepsilon'>0$ and effective numbers
$k_0,c>0$ such that for any positive integers $k\geq k_0$ and
$k^{2-\varepsilon} \le N\le  k^{2+\varepsilon}$, there is an explicit $n \times N$ RIP matrix of order
$k$ with $k\ge c n^{1/2+\eps/4}$ and constant $\delta=k^{-\varepsilon'}$.
\end{thm}

As of this writing, the constructions in \cite{BDFKK} and \cite{STOC} 
remain the only explicit constructions of RIP matrices which 
exceed the $\sqrt{n}$ barrier for $k$.

The proof of Theorem \ref{thm1} depends on two key results in additive combinatorics.
For subsets $A,B$ of an additive finite group $G$, we write
\begin{align*}
  A\pm B &= \{a \pm b: a\in A,b\in B\},
  \\ E(A,B) &= \# \{(a_1,a_2,b_1,b_2) : a_1+b_1=a_2+b_2; a_1,a_2\in A; b_1,b_2\in B\}.
\end{align*}
Also set $x \cdot B = \{xb:b\in B\}$.   Here we will mainly work with the group
of residues modulo a prime $p$.

\begin{prop}\label{TheoremC}
For some $c_0$, the following holds.
 Assume $A,B$ are subsets of residue classes modulo $p$, with $0\not\in B$
  and $|A|\ge |B|$.  Then
\begin{equation}\label{bour1}
\sum_{b\in B} E(A, b\cdot A) = O\( \left( \min(p/|A|,|B| \right)^{-c_0}|A|^3|B| \).
\end{equation}
\end{prop}

This theorem, without an explicit $c_0$, was proved by Bourgain \cite{Bo}.
The first explicit version of Proposition \ref{TheoremC}, with $c_0=1/10430$,
is given in Bourgain and Glibuchuk \cite{BGlib}, and this is the value
used in the papers \cite{BDFKK,STOC}.
Murphy and Petridis \cite[Lemma 13]{MP} made a great improvement, showing
that Proposition \ref{TheoremC} holds with $c_0=1/3$.
It is conceivable that $c_0$ may be taken to be any number less than 1.
Taking $A=B$ we see that $c_0$ cannot be taken larger than 1.

We also need a version of the
 Balog--Szemer\'edi--Gowers lemma, originally proved by Balog and Szemer\'edi \cite{BS} and later improved by Gowers \cite{G}. The version we use is a later improvement due to Schoen \cite{Schoen}.
 
\begin{prop}\label{BSG}
For some positive $c_1, c_2, c_3$ and $c_4$, the following holds.
If $E(A,A) = |A|^3/K$, then there exists $A', B' \subseteq A$
with $|A'| , |B'| \ge c_2 \frac{|A|}{K^{c_4}}$ and $|A'-B'| \le c_3 K^{c_1} |A'|^{1/2} |B'|^{1/2}$.
\end{prop} 

The constants $c_2,c_3$ are relatively unimportant.
The best result to date is due to Schoen \cite{Schoen}, who showed that any 
$c_1 > 7/2$ and $c_4 > 3/4$ is admissible.  It is conjectured that $c_1=1$ is admissible.
The papers \cite{BDFKK,STOC} used Proposition~\ref{BSG} with the weaker values $c_1=9$ and $c_4 = 1$, this
deducible from Bourgain and Garaev \cite[Lemma 2.2]{BG}.

\section{Construction of the matrix}\label{sec:construction}
%
%

Our construction is identical to that in \cite{STOC}.
We fix an even integer $m\ge 100$ and let $p$ be a large prime.
For $x\in \ZZ$, let $e_p(x)=e^{2\pi i x/p}$.
Let
\be\label{uvec}
{\mathbf u}_{a,b}=\frac1{\sqrt p}(e_p(ax^2+bx))_{1\le x\le p}.
\ee
We take
\begin{equation}\label{Adef}
\alpha=\frac{1}{2m},
\qquad \curly A=\{1, 2, \ldots \lfloor p^{\alpha}\rfloor\}.
\end{equation}
To define the set $\curly B$, we take
\[
\beta=\frac1{2.01m},\quad r=\left\lfloor\frac{\beta\log p}
{\log 2}\right\rfloor, \quad M=\lfloor 2^{2.01m-1} \rfloor,
\]
and let
\begin{equation}\label{Bdef}
\curly B=\left\{\sum_{j=1}^r x_j(2M)^{j-1}:\,x_1,\dots,x_r\in\{0,\dots,M-1\}
\right\}.
\end{equation}
We interpret $\curly A, \curly B$ as sets of residue classes modulo $p$.
We notice that all elements of $\curly B$ are at most $p/2$, 
and 
$|\curly A||\curly B|$
lies between two constant multiples of $p^{1+\alpha - \beta} = p^{1+1/(402m)}.$

Given large $k$ and $k^{2-\varepsilon} \le N \le k^{2+\varepsilon}$, let $p$ be a prime
in the interval $[k^{2-\varepsilon},2k^{2-\varepsilon}]$ (such $p$ exists by Bertrand's postulate).
Let $\Phi_p$ be a $p \times (|\curly A|\cdot|\curly B|)$ matrix formed by the 
column vectors $\mathbf{u}_{a,b}$ for $a\in \curly A, b\in \curly B$ (the columns
may appear in any order).
We also have
\be\label{epsm}
\text{ if } \eps \le \frac{1}{403m}, \text{ then }
N \le p^{\frac{2+\eps}{2-\eps}} \le  |\curly A||\curly B|.
\ee
Take $\Phi$ to be the
matrix formed by the first $N$ columns of $\Phi_p$.
Let $n=p$.
Our task is to show that $\Phi$ satisfies the RIP condition
with $\delta=p^{-\eps'}$ for some constant $\eps'>0$, and of order $k$.

%
\section{Main tools}\label{Sec3}
%

\begin{lem}\label{main} 
Assume that $c_0 \le 1$ and that
Proposition \ref{TheoremC} holds.
Fix an even integer $m\ge 100$, and define $\alpha,\curly A,\curly B$
by \eqref{Adef} and \eqref{Bdef}.
Suppose that $p$ is sufficiently large in terms
of  $m$. 
Assume also that for some constant $c_5>0$ and constant $0<\gamma \le \frac{1}{4m}$, $\curly B$ satisfies
\be\label{ESS}
\forall\; S\subseteq \curly B \text{ with } |S|\ge p^{0.49}, \;\;
E(S,S) \le c_5 p^{-\gamma} |S|^3.
\ee
Define the vectors $\mathbf{u}_{a,b}$ by \eqref{uvec}.
Then for any disjoint sets
$\Omega_1,\Omega_2\subset\curly A\times\curly B$ such that
$|\Omega_1|\le\sqrt p$, $|\Omega_2|\le\sqrt p$, the inequality
$$
\left|\sum_{(a_1,b_1)\in \Omega_1}\sum_{(a_2,b_2)\in \Omega_2}
\left\langle {\mathbf u}_{a_1,b_1}, {\mathbf u}_{a_2,b_2}\right\rangle\right|
=O\( p^{1/2-\varepsilon_1} (\log p)^2\)
$$
holds, where
\begin{equation}\label{eps1}
 \varepsilon_1 = \frac{\frac{c_0\gamma}{8} - \frac{47\alpha-23\gamma}{2m}}{1+93/m + c_0/2}.
\end{equation}
The constant implied by the $O$-symbol depends only on $c_0, \gamma$ and $m$.
\end{lem}

Lemma \ref{main} follows by combining Lemmas 2 and 4 from \cite{STOC}; the 
assumption of Proposition \ref{TheoremC} is inadvertently omitted in the statement
of \cite[Lemma 4]{STOC}.

Using Lemma \ref{main}, we shall show the following.

\begin{thm}\label{thm2}
Assume the hypotheses of Lemma \ref{main}, let
$\eps=2\eps_1-2\eps_1^2$ and assume that $\eps \le \frac{1}{403m}$.  There is $\eps'>0$ such that
for sufficiently large $k$ and
$k^{2-\varepsilon} \le N\le  k^{2+\varepsilon}$, there is an explicit $n \times N$ RIP matrix of order
$k$ with $n=O(k^{2-\varepsilon})$ and constant $\delta=k^{-\varepsilon'}$.
\end{thm}

To prove Theorem \ref{thm2}, we first recall another additive combinatorics result
from \cite{STOC}.

\begin{lem}[{\cite[Theorem 2, Corollary 2]{STOC}}]\label{cor2} Let $M$ be a positive integer.
For the set $\curly B\subset{\mathbb F_p}$ defined in
\eqref{Bdef} and for any subsets $A,B\subset\curly B$, we have
$|A-B|\ge |A|^{\tau} |B|^{\tau}$, where $\tau$ is the unique positive
solution of 
\[
\left(\frac1M\right)^{2\tau}+\left(\frac{M-1}M\right)^{\tau}=1.
\]
\end{lem}

From \cite{STOC} we have the easy bounds
\be\label{tau-bounds}
\frac{\log 2}{\log M} \(1-\frac{1}{\log M}\) \le 2\tau -1 \le \frac{\log 2}{\log M}.
\ee

\begin{cor}\label{lem3}
Take $\curly B$ as in \eqref{Bdef} and assume Proposition \ref{BSG}.  Then \eqref{ESS} holds with
\[
\gamma = \frac{0.49(2\tau-1)}{c_1+c_4(2\tau-1)}.
\]
\end{cor}

\begin{proof}
Just like the proof of \cite[Lemma 3]{STOC}, except that we incorporate
Proposition \ref{BSG}.
Suppose that $S\subseteq \curly B$ with $|S|\ge p^{0.49}$ and
 $E(S,S)=|S|^3/K.$  By Proposition
\ref{BSG}, there are sets $T_1 , T_2\subset S$ such that $|T_1|, |T_2|\ge c_2 \frac{|S|}{K^{c_4} }$
and $|T_1-T_2| \le c_3 K^{c_1} |T_1|^{1/2} |T_2|^{1/2}$.  By Lemma \ref{cor2},
\[
c_3 K^{c_1} |T_1|^{1/2} |T_2|^{1/2} \ge |T_1-T_2| \ge |T_1|^{\tau}|T_2|^{\tau}  ,
\]
and hence
\[
c_3 K^{c_1} \ge \big( |T_1|\cdot |T_2|\big)^{\tau-1/2} \ge \pfrac{c_2 p^{0.49}}{K^{c_4}}^{2\tau -1}.
\]
It follows that $K\ge (1/c_5) p^{-\gamma}$ for an appropriate constant $c_5>0$.
\end{proof}

Finally, we need a tool from \cite{BDFKK} which states 
that  in \eqref{eq: RIP} we need only consider vectors ${\mathbf x}$ whose components
are 0 or 1 (so-called \emph{flat} vectors).

\begin{lem}[{\cite[Lemma 1]{BDFKK}}]\label{flat2} Let $k\ge2^{10}$ and $s$ be a positive integer.
Assume that for all $i\ne j$ we have $\langle \mathbf{u}_i, \mathbf{u}_j \rangle \le 1/k$.
Also, assume that for some $\delta\ge0$ and any disjoint
$J_1,J_2\subset\{1,\dots,N\}$ with $|J_1|\le k, |J_2|\le k$ we have
$$\left|\left\langle\sum_{j\in J_1}{\mathbf u}_j,\sum_{j\in J_2}{\mathbf u}_j\right
\rangle\right| \le\delta k.
$$
Then $\Phi$ satisfies the RIP property of order $2sk$ with constant
$44s\sqrt{\delta}\log k$.
\end{lem}

Now we show how to deduce Theorem \ref{thm2}.
By Lemma \ref{main} and standard bounds for Gauss sums,
$\Phi$ satisfies the conditions of Lemma
\ref{flat2} with
$k=\lfloor \sqrt{p}\rfloor$ and
$\delta=O(p^{-\varepsilon_1}\log^2 p)$.
Let $\varepsilon_0<\varepsilon_1/2$ and take $s=\lfloor p^{\varepsilon_0} \rfloor$.
By Lemma \ref{flat2}, $\Phi$ satisfies RIP with order
$\ge p^{1/2+\varepsilon_0}$ and constant $O(p^{-\varepsilon_1/2+\varepsilon_0}(\log p)^3)$.
If $\varepsilon_0$ is sufficiently close to $\varepsilon_1/2$, Theorem \ref{thm2} follows with
$$
\varepsilon = 2 - \frac{2}{1+2\varepsilon_0} = \frac{4\eps_0}{1+2\eps_0} >
2\eps_1 - 2\eps_1^2.
$$

To prove Theorem \ref{thm1}, we take the construction in Section \ref{sec:construction}.
We have \eqref{ESS} by Corollary \ref{lem3}.
Also take
\[
\eta = 10^{-100}, \qquad c_0 = \frac13, \qquad c_1=7/2+\eta, \qquad c_4 = 3/4+\eta, \qquad m=7586.
\]
These values were optimized with a computer search.
By Corollary \ref{lem3} and \eqref{tau-bounds}, we have $\gamma \ge 9.182 \cdot 10^{-6}$.
It is readily verified that $\gamma \le \frac{1}{4m}$, $\eps_1 > 1.631 \cdot 10^{-7}$
and $\eps=2\eps_1 - 2\eps_1^2$ satisfies $3.26\cdot 10^{-7} \le \eps \le \frac{1}{403 m}$.
Theorem \ref{thm1} now follows.

\section{Acknowledgments}
The first author was partially supported by NSF Grant DMS-1802139.
The second author is supported by a Simons Travel grant. The third author is supported by Ben Green's Simons Investigator Grant 376201.


\end{document}